\documentclass[12pt]{amsart}
\usepackage{fullpage}
\usepackage{amssymb,latexsym}
\usepackage{enumerate}
\usepackage{hyperref}

\makeatletter \@namedef{subjclassname@2010}{%
  \textup{2010} Mathematics Subject Classification}
\makeatother

\newcounter{thm} 
\newtheorem{Theorem}[thm]{Theorem}

\newtheorem{Proposition}[thm]{Proposition}
\newtheorem{Lemma}[thm]{Lemma}

\newtheorem*{Problem}{Problem}

\newcommand{\FF}[0]{\mathbb F}

	\newcommand{\ZZ}[0]{\mathbb Z}

	\newcommand{\cR}[0]{\mathcal R}

\renewcommand{\aa}[0]{\textbf{\textit{a}}}

\newcommand{\xx}[0]{\textbf{\textit{x}}}

\newcommand{\zz}[0]{\textbf{\textit{z}}}

\newcommand{\eps}[0]{\varepsilon}

\renewcommand{\pmod}[1]{\text{ (mod }#1)}

\newcommand{\lr}[1]{\left(#1\right)}

\renewcommand{\phi}{\varphi}

\begin{document}


\baselineskip=17pt



\title{The Density of Numbers Represented by Diagonal Forms of Large Degree}

\author[Brandon Hanson]{Brandon Hanson} \address{Pennsylvania State University\\
University Park, PA}
\email{bwh5339@psu.edu}

\author[Asif Zaman]{Asif Zaman} \address{University of Toronto\\
Toronto, ON}
\email{asif@math.toronto.edu}
\maketitle

\begin{abstract}
	Let $s \geq 3$ be a fixed positive integer and $a_1,\dots,a_s \in \ZZ$ be arbitrary. We show that, on average over $k$, the density of numbers represented by the degree $k$ diagonal form
	\[
	a_1 x_1^k + \cdots + a_s x_s^k
	\]
	decays rapidly with respect to $k$. 	
\end{abstract}

\section{Introduction}
The classical version of Waring's problem asks whether every positive integer can be written as a sum of at most $s$ positive integers, each of which is a $k$'th power. In other words, is there an integer $s$ (which depends on $k$) such that for each $n\geq 1$ we have a solution to the equation
\begin{equation}
n=x_1^k+\cdots +x_s^k
\label{eqn:Waring}	
\end{equation}
in non-negative integers $x_1,\ldots,x_s$? The least value of $s$ which is admissible is usually referred to as $g(k)$, and Waring's problem is thus the assertion that $g(k)<\infty$ for any $k\geq 1$. Waring's problem has a long history; for a nice exposition see \cite{VW}.

The ``easier'' version of Waring's problem, a name attributed to Wright \cite{wright1934}, asks whether there is a solution to the equation
\begin{equation}
n= x_1^k\pm \cdots \pm x_s^k.
\label{eqn:Waring_easy}	
\end{equation}
The least $s$ for which this equation is soluble for each $n$ is usually referred to as $v(k)$, and establishing that $v(k)<\infty$ is a fairly simple argument, which can be found in \cite{N}. Clearly any upper bound for $g(k)$ in the usual Waring problem extends to a bound for $v(k)$ as well. However, the freedom to use negative summands may make $v(k)$ considerably smaller.

One can verify that $g(k)\geq 2^k-1$. Indeed, in order for $2^k-1$ to be written as a sum of $k$'th powers, we only have $1$'s at our disposal. For these reasons, one usually considers instead $G(k)$, which is the least $s$ such that \eqref{eqn:Waring} is soluble for all $n$ sufficiently large. Here, the bound $G(k)\geq k$ is still quite simple. To represent each $n$ in the range $X\leq n\leq 2X$, the variables $x_i$ can be no larger than $X^{1/k}$. Thus the vector $(x_1,\ldots,x_s)$ is a lattice point in the box $[0,X^{1/k}]^s$, and there are at most $O(X^{s/k})$ such lattice points. To represent all integers in the desired range, we must therefore have $s\geq k$. The introduction of negative summands causes this argument to fail completely, because one is no longer counting lattice points in a bounded region. This motivates the following question, which was asked in \cite{BB}:

\begin{Problem}
For $k$ sufficiently large, is it true that the set of integers of the form
\[n=x_1^k\pm x_2^k\pm  \cdots \pm x_s^k\]
has asymptotic density zero?
\end{Problem}

A result of Wooley (see \cite{BB} and \cite{W}) asserts that, for $k\gg s^3$, the set of integers of the form
\[n=x_1^k\pm \cdots \pm x_s^k\] has density zero and in fact more is true -- one can obtain fairly good decay rates in the proportion of integers up to $X$ which can be represented. However, Wooley's result is conditional on a generalized version of the $abc$-conjecture and, as far as the authors are aware, there seems to be little known unconditionally for large values of $s$, say $s \geq 5$. We prove a result in this direction which is much weaker, but unconditional. We will not be able to prove that the set of integers represented has zero density, but we will establish bounds on the asymptotic density of these integers. These bounds will, on average, decay quite rapidly with respect to $k$. 

In fact, we will establish something a bit more general in that we will allow for arbitrary integer coefficients, not just $1$'s and $-1$'s. Let $s\geq 2$ be fixed and let $\aa=(a_1,\ldots,a_s)\in\ZZ^s$ be arbitrary. We consider the form 
\begin{equation}
F_{\aa,k}(\xx)=a_1x_1^k+\cdots+a_sx_s^k
\label{eqn:form}	
\end{equation}
and the set
\[\cR(\aa,k)=\{n:n=F_{\aa,k}(\xx)\text{ for some }\xx\in\ZZ^s\}\] of numbers which this form represents. 

We shall estimate the average asymptotic (upper)-density 
\begin{equation}
\delta_k=\limsup_{N\to \infty} \frac{|\cR(\aa,k)\cap [1,N]|}{N}
\label{def:delta_k}	
\end{equation}
as a function of $k$. This number implicitly depends on $\aa$, but the results we shall prove about $\delta_k$ are uniform over $\aa$. For $s \geq 3$, the following theorem establishes that the value of $\log(1/\delta_k)$ is large on average. 

\begin{Theorem} \label{MainTheorem} Let $s \geq 3$ be fixed and $\aa \in \ZZ^s$ be arbitrary. Let $X \geq 3$ be sufficiently large depending at most on $s$ and define $\delta_k$ as in \eqref{def:delta_k}. Then
\begin{equation}
\frac{1}{X}\sum_{1 \leq k < X} \log(1/\delta_k) \gg \frac{X^{\frac{1}{s-1}}}{\log X}.
\label{eqn:THM1}
\end{equation}
\end{Theorem}
We will use the convention that $\log(1/\delta_k) = \infty$ when $\delta_k = 0$. Thus, we expect that the quantity on the lefthand side in \eqref{eqn:THM1} is infinite for all $X$ sufficiently large depending only on $s$ and $\aa$. Perhaps it is instructive to compare Theorem \ref{MainTheorem} to a conditional result. Let $\pi(X;k,1)$ denote the number of primes $p < X$ satisfying $p \equiv 1 \pmod{k}$ and let $\varphi(k)$ denote Euler's totient function. 
\begin{Proposition} \label{MainProp}
Let $s \geq 3$ be fixed and $\aa \in \ZZ^s$ be arbitrary. Let $k \geq 3$ be sufficiently large depending at most on $s$ and define $\delta_k$ as in \eqref{def:delta_k}. If
\begin{equation}
\pi(X;k,1) = \frac{1}{\varphi(k)} \mathrm{Li}(X) + O_{\eps}\Big( \frac{X^{1/2+\eps}}{k^{1/2}} \Big)
\label{eqn:PrimesinAPs_conjecture}
\end{equation}
for any $\eps > 0$ and $X \geq k^{1+\eps}$ then 
\begin{equation}
\log(1/\delta_k) \gg \frac{k^{\frac{1}{s-1}}}{\log k}.
\label{eqn:PrimesinAPs_conjecture_result}
\end{equation}
\end{Proposition}
 Assumption \eqref{eqn:PrimesinAPs_conjecture} is one of the strongest widely-believed conjectures regarding the distribution of primes in arithmetic progressions and Theorem \ref{MainTheorem} unconditionally obtains the corresponding average result for $\log(1/\delta_k)$.   
Note that the special case $s=2$ is addressed by classical work of Mahler \cite{Mahler1933} which implies $\delta_k = 0$ for $k \geq 3$. For further details on the case $s=2$, see for example \cite{BDW1998} and \cite{stewart-xiao2016}. 

\section{Local densities and a conditional result}
 The method of proof for Theorem \ref{MainTheorem} and Proposition \ref{MainProp} is to bound the density $\delta_k$ by considering local constraints. For instance, a very simple first observation is that for $p$ prime
\[\delta_{p-1}\leq \frac{2^s}{p}.\] This is just the trivial observation that the set of $(p-1)^{\text{th}}$ powers modulo $p$ consist of the residue classes $0$ and $1$ modulo $p$. Thus there are at most $2^s$ admissible values of the diagonal form $F_{\aa,k}$ modulo $p$. Our aim is then to improve this estimate for a given $k$, and subsequently obtain good density estimates for the average exponent $k$. To this end, define
\[\delta_k(p)=\frac{1}{p}\left|\{F_{\aa,k}(\zz)\mod p:\zz\in\FF_p^s\}\right|.\]
The Chinese Remainder Theorem then gives:
\begin{Lemma}
\label{lem:CRT}
For any integer $k \geq 1$,
\[\delta_k\leq \prod_{p}\delta_k(p).\]
\end{Lemma}
We will combine this with a simple development of the idea we used to bound $\delta_{p-1}$. Let $(m,n)$ denote the greatest common divisor of two integers $m$ and $n$. 
\begin{Lemma}
\label{lem:Local_Density}
Let $k\geq 2$ and let $p$ be a prime. Then
\[\delta_k(p) \leq \alpha_{k,p}\] where \[\alpha_{k,p}=\frac{1}{p} \lr{\frac{p-1}{(k,p-1)} + 1}^s.\]
\end{Lemma}
\begin{proof}
By the structure theorem for cyclic groups, the set non-zero $k^{\text{th}}$ powers modulo $p$ forms a subgroup of the unit group of size $\frac{p-1}{(k,p-1)}$. Adding in the $0$ class modulo $p$, there are $\frac{p-1}{(k,p-1)}+1$ values of $z^k$ modulo $p$. Thus, the proportion of admissible residue classes modulo $p$ is at most $\alpha_{k,p}$. 
\end{proof}

Using these two lemmas, we can establish Proposition \ref{MainProp}.  

\begin{proof}[Proof of Proposition \ref{MainProp}]
Suppose $p$ is a prime satisfying $p \equiv 1 \pmod{k}$. With $\alpha_{k,p}$ defined as in Lemma \ref{lem:Local_Density}, observe that $\alpha_{k,p} < 1$ if and only if
\[
\frac{p-1}{k} + 1 < p^{1/s}.
\]
Since $\frac{p-1}{k} + 1 < \frac{p}{k} + \frac{p}{k} = \frac{2p}{k}$, we have that $\alpha_{k,p} < 1$ whenever
\[
p < \tfrac{1}{4}k^{1/(1-1/s)} = \tfrac{1}{4}k^{1 + 1/(s-1)}. 
\]
Let $R \geq 10$ be sufficiently large and set $Z = \tfrac{1}{R}k^{1 + 1/(s-1)}$. Thus, by Lemmas \ref{lem:CRT} and \ref{lem:Local_Density},
\begin{align*}
\log(1/\delta_k) \geq \sum_p \log(1/\delta_k(p)) 
& \geq \sum_{\substack{p \equiv 1 \pmod{k} \\ p < Z} }\lr{\log p - s \log\lr{ \frac{p-1}{k} + 1} }\\
& \geq \sum_{\substack{p \equiv 1 \pmod{k} \\ p < Z} }\lr{s \log k - (s-1) \log p + O(1)}.
\end{align*}
Since $p < Z$, the above is
\[
\geq \sum_{\substack{p \equiv 1 \pmod{k} \\ p < Z} }\lr{(s-1) \log R + O(1)} \gg \sum_{\substack{p \equiv 1 \pmod{k} \\ p < Z}} 1,
\]
after fixing $R$ to be sufficiently large, depending at most on $s$. 
Then, using assumption \eqref{eqn:PrimesinAPs_conjecture} to bound the sum on the right in the above inequality, it follows that \
\[\log(1/\delta_k)\gg \frac{k^{1+1/(s-1)}}{\phi(k)\log k} \gg \frac{k^{1/(s-1)}}{\log k},\] after bounding $\phi(k)$ trivially by $k$. This proves the proposition. 
\end{proof}

\section{Global density is small on average}

This section is dedicated to proving Theorem \ref{MainTheorem} for which we require one additional lemma. For integers $a$ and $q$, let 
\begin{equation}
\psi(X; q,a) = \sum_{\substack{n < X \\ n \equiv a \pmod{q}}} \Lambda(n),
\label{def:Psi}
\end{equation}
where $\Lambda(n)$ equals $\log p$ if $n$ is a power of a prime $p$ and equals $0$ otherwise. 

\begin{Lemma} \label{lem:MainTheoremPrep}
Let $X \geq 1$ be arbitrary. For $Y \leq X^{1/2}$,  
\begin{equation}
\sum_{1 \leq m < Y} \psi(mX; m, 1) 
	 \geq \frac{\zeta(2) \zeta(3) \log 2}{\zeta(6)} XY + O(X \log^2 X + XY (\log X)^{-2}).
	 \label{eqn:Lem3-1}
\end{equation}
Additionally, if $s \geq 2$ is fixed then
\begin{equation}
\sum_{1 \leq m < Y} \int_{(m+1)^s}^{mX} \dfrac{\psi(t; m,1)}{t (\log t)^2 } dt 
	\ll \frac{XY}{\log^2 X}.
	\label{eq:Lem3-2}
\end{equation}
\end{Lemma}
\begin{proof}
We divide the sum in \eqref{eqn:Lem3-1} dyadically. For $1 \leq M \leq 2M \leq X^{1/2}$, note $M \leq (MX)^{1/3}$. Hence, by the Bombieri-Vinogradov theorem \cite[Theorem 17.1]{iwanieckowalski}, we have  that
\begin{align*}
\sum_{M \leq m < 2M} \psi(mX; m, 1) 
	& \geq \sum_{M \leq m < 2M} \psi(MX; m, 1) \\
	& = \Big( \sum_{M \leq m < 2M} \frac{MX}{\varphi(m)} \Big) + O( MX (\log MX)^{-2} ) \\
	& = \frac{\zeta(2) \zeta(3) \log 2}{\zeta(6)} MX + O(X \log M + MX (\log MX)^{-2} ).
\end{align*}
In the last step, we applied the classical fact  \cite{landau1900ueber} that, for $x \geq 2$, 
\[
\sum_{n \leq x} \frac{1}{\varphi(n)} = \frac{\zeta(2)\zeta(3)}{\zeta(6)} \Big( \log x + \gamma - \sum_p \frac{\log p}{p^2-p+1} \Big) + O(x^{-1} \log x). 
\]
Summing the prior estimate over $M = Y/2^{j+1}$ with $0 \leq j \leq \lfloor \log Y \rfloor$ and recalling $Y \leq X^{1/2}$ yields desired result. To prove \eqref{eq:Lem3-2}, we proceed similarly. For $t \geq (m+1)^s$ and $s \geq 2$, we may apply the Brun-Titchmarsh inequality \cite[Theorem 6.6]{iwanieckowalski} to $\psi(t;m,1)$ and deduce that
\begin{align*}
\sum_{M \leq m < 2M} \int_{(m+1)^s}^{mX} \dfrac{\psi(t; m,1)}{t (\log t)^2 } dt
& \ll \sum_{M \leq m < 2M} \int_{(m+1)^s}^{mX} \dfrac{1}{\varphi(m) (\log t)^2 } dt\\
& \ll \sum_{M \leq m < 2M} \frac{1}{\varphi(m)} \int_{(M+1)^s}^{2MX} \dfrac{1}{(\log t)^2 } dt\\
	& \ll  \sum_{M \leq m < 2M} \frac{MX}{\varphi(m) \log^2(MX)} \\ 
    & \ll \frac{MX}{\log^2 X}.
\end{align*}
The desired bound follows by dyadically summing this estimate. 
\end{proof}

\begin{proof}[Proof of Theorem \ref{MainTheorem}] 
By Lemma \ref{lem:CRT}, observe that
\[
\sum_{1 \leq k < X} \log(1/\delta_k) \geq S_1\] where 
\[S_1=\sum_{1 \leq k < X} \sum_p \log(1/\delta_k(p)).\]
It suffices to show $S_1 \gg X^{1+1/(s-1)}/\log X$. By Lemma \ref{lem:Local_Density}, 
\[
\delta_k(p) \leq \min\{\alpha_{k,p}, 1\}
\]
for each prime $p$. Let $m = \frac{p-1}{(k,p-1)}$, so $k = \frac{p-1}{m} \cdot d$ for some integer $d \geq 1$ with $(d,p-1) = 1$. Thus, each pair $(k,p)$ in the above sum corresponds to a unique triple $(m,d,p)$ of positive integers $m$ and $d$ such that $m \mid p-1$ and $(d,p-1) = 1$. Moreover,
\[
\alpha_{k,p} < 1 \iff p > (m+1)^s.
\]
Collecting these observations, we deduce that
\begin{equation*}
\begin{aligned}
	S_1 & \geq \sum_{m \geq 1} \sum_{\substack{m \mid p-1 \\ p > (m+1)^s} } \sum_{\substack{1 \leq d \leq \frac{mX}{p-1} \\ (d,p-1) = 1}}  \big(\log p - s \log(m+1) \big).  
\end{aligned}
\end{equation*}
Whenever $p < mX$, the inner sum over $d$ contains $d = 1$. Thus, the above is 
\begin{equation*}
\begin{aligned}
	& \geq \sum_{m \geq 1} \sum_{\substack{p \equiv 1 \pmod{m} \\ (m+1)^s < p < mX} } \big(\log p - s \log(m+1) \big).  
\end{aligned}
\end{equation*}
By positivity, we may restrict the outer sum to $1 \leq m < Y$ for some parameter $Y \geq 1$  satisfying
\begin{equation}
(Y+1)^s < XY.
\label{eqn:Y-condition}
\end{equation}
Recalling \eqref{def:Psi}, it follows by partial summation that 
\begin{equation}
\label{eqn:S_1-reduced}
\begin{aligned}
	S_1 & \geq \sum_{1 \leq m < Y} \int_{(m+1)^s}^{mX} \Big( 1 - \frac{s \log(m+1)} {\log t}\Big) d\psi(t; m,1) \\
    & \geq  \sum_{1 \leq m < Y}  \Big( \psi(mX; m, 1) \big(1-\frac{s \log(m+1)}{\log(mX)}\big) - s \log(m+1) \int_{(m+1)^s}^{mX} \dfrac{\psi(t; m,1)}{t (\log t)^2 } dt \Big). 
\end{aligned}
\end{equation}
Set $Y = X^{1/(s-1+\eta)}$ where $\eta = \eta(X) < 1/2$ is a parameter which will be specified. For $1 \leq m < Y$, we have that
\begin{equation}
1- \frac{s\log(m+1)}{\log(mX)}  = 1 - \frac{s}{1+\frac{\log X}{\log m}} + O(\frac{1}{m \log X}).
\label{eqn:Weight}
\end{equation}
If $X^{1/(2s-1)} \leq m \leq X^{1/(s-1+\eta)}$, then from \eqref{eqn:Weight} we see that
\[
1- \frac{s\log(m+1)}{\log(mX)} \geq \frac{\eta}{s} + O(X^{-1/(2s-1)}) \geq \frac{\eta}{s} + O(\frac{1}{\log X}). 
\]
Otherwise, for $m \leq X^{1/(2s-1)}$, we similarly have that
\[
1- \frac{s\log(m+1)}{\log(mX)} \geq \frac{1}{2} + O(\frac{1}{\log X}) \geq \frac{\eta}{s} + O(\frac{1}{\log X}).
\]
Substituting these bounds into \eqref{eqn:S_1-reduced} and noting $s \log(m+1) < s \log(Y+1) \ll \log X$, we deduce that
\[
S_1 \geq \lr{\sum_{1 \leq m < Y} \psi(mX; m, 1)}\lr{\frac{\eta}{s} + O\lr{\frac{1}{\log X}}} + O\lr{\log X \sum_{1 \leq m < Y} \int_{(m+1)^s}^{mX} \dfrac{\psi(t; m,1)}{t (\log t)^2 } dt}.
\]
Since $Y \leq X^{1/2}$ for $s \geq 3$, Lemma \ref{lem:MainTheoremPrep} therefore implies that
\[
S_1 \geq \frac{\zeta(2)\zeta(3)\log 2}{\zeta(6)} \cdot \frac{\eta}{s} \cdot XY + O\lr{\frac{XY}{\log X}}.
\]
Note the implied constant is independent of $\eta$ and depends only on $s$. Choose $\eta = C/\log X$ where $C$ is a fixed sufficiently large constant depending only on $s$. Thus, $Y = X^{1/(s-1+\eta)} \gg X^{1/(s-1)}$ satisfies \eqref{eqn:Y-condition} and it follows that $S_1 \gg XY/\log X \gg X^{1+1/(s-1)}/\log X$. 
\end{proof}

\bibliographystyle{plain}
\bibliography{density}

\end{document}